\documentclass[final]{amsart}
\usepackage{xcolor,soul,cancel}
\usepackage[color]{showkeys}
\usepackage{tikz}
\usetikzlibrary{positioning}
\definecolor{refkey}{rgb}{0,1,1}
\definecolor{labelkey}{rgb}{1,0,0}
\usepackage{amssymb}
\usepackage{amsmath}
\usepackage{amsthm}
\usepackage{graphicx}

\newtheorem{theorem}{Theorem}
\newtheorem{lemma}[theorem]{Lemma}

\theoremstyle{definition}

\newtheorem{cor}{Corollary}
\theoremstyle{remark}

\numberwithin{equation}{section}

\newcommand{\abs}[1]{\lvert#1\rvert}
\newcommand{\norm}[1]{\left\Vert#1\right\Vert}
\newcommand{\scal}[1]{\langle#1\rangle}
\newcommand{\conv}{\operatorname{conv}}
\newcommand{\re}{\operatorname{Re}}

\newcommand{\eq} [1] {\begin{equation}\label{#1}\quad}
\newcommand{\en} {\end{equation}}

\begin{document}

\title[Anderson's theorem revisited]{A note on Anderson's theorem \\
in the infinite-dimensional setting}

\author[Birbonshi]{Riddhick Birbonshi}
\address{Department of Mathematics,
Indian Institute of Technology Kharagpur,
Kharagpur 721302, India}
\email{riddhick.math@gmail.com}

\author[Spitkovsky]{Ilya M. Spitkovsky}
\address{Division of Science,  New York  University Abu Dhabi (NYUAD), Saadiyat Island,
P.O. Box 129188 Abu Dhabi, UAE}
\email{ims2@nyu.edu, imspitkovsky@gmail.com}
\thanks{Supported in part by Faculty Research funding from the Division of Science and Mathematics, New York University Abu Dhabi.}

\author[Srivastava]{P. D. Srivastava}
\address{Department of Mathematics,
Indian Institute of Technology Kharagpur,
Kharagpur 721302, India}
\email{pds@maths.iitkgp.ernet.in}

\date{}



\begin{abstract}Anderson's theorem states that if the numerical range $W(A)$ of an $n$-by-$n$ matrix $A$ is contained in the unit disk $\overline{\mathbb D}$ and intersects with the unit circle at more than $n$ points, then $W(A)=\overline{\mathbb D}$. An analogue of this result for compact $A$ in an infinite dimensional setting was established by Gau and Wu. We consider here the case of $A$ being the sum of a normal and compact operator.
\end{abstract}

\maketitle

\section{Introduction}

The {\em numerical range} (also known as the {\em field of values}, or the {\em Hausdorff set}) of a bounded linear operator $A$
acting on a Hilbert space $\mathcal H$ is defined as
\[ W(A)=\{\scal{Ax,x}\colon \norm{x}=1\}. \]
Here $\scal{.,.}$ and $\norm{.}$ stand for the scalar product on $\mathcal H$ and the norm generated by it, respectively.

The set $W(A)$ is a convex (Toeplitz-Hausdorff theorem), bounded, and in the case $\dim{\mathcal H}<\infty$ also closed subset of the complex plane $\mathbb C$.

We will use the standard notation $\overline{X}, X^o, \partial X, X'$ for the closure, interior, the boundary, and the set of the limit points, respectively, of subsets $X\subset\mathbb C$.
In particular, ${\mathbb D}=\{ z\colon \abs{z}<1\}$ is the open unit disk, $\partial{\mathbb D}=\mathbb T$ is the unit circle, and $\overline{\mathbb D}={\mathbb D}\cup\partial\mathbb D$
is the closed unit disk.

The closure $\overline{W(A)}$ of the numerical range of $A$ contains the spectrum $\sigma(A)$, and thus the convex hull $\conv\sigma(A)$ of the latter. For normal $A$, $\overline{W(A)}=\conv\sigma(A)$. We refer to \cite{GusRa} for these and other well known properties of the numerical range.

Anderson's theorem (unpublished by the author but discussed e.g. in \cite{GauWu03,RadHad})
states that if $W(A)$ is contained in $\overline{\mathbb D}$ and the intersection of $W(A)$ with $\mathbb T$ consists of more than $n=\dim\mathcal H$ points,
then in fact $W(A)=\overline{\mathbb D}$. This result is sharp in a sense that for a unitary operator $U$ with a simple spectrum acting on an $n$-dimensional $\mathcal H$, $W(U)$ is a polygon with $n$ vertices on $\mathbb T$
and thus different from $\overline{\mathbb D}$.

Unitary diagonal operators also deliver easy examples showing that Anderson's theorem does not generalize to the infinite-dimensional setting. Indeed, if $A$ is a diagonal operator with the point spectrum
$\sigma_p(U)=\{\lambda_j, \ j=1,2,\ldots\}\subset\mathbb T$, then $\overline{W(A)}=\conv\sigma_p(A)\varsubsetneq\overline{\mathbb D}$ while $W(A)\cap{\mathbb T}=\sigma_p(A)$ is infinite.

Moreover, according to \cite{RadRad} every bounded convex set $G$ for which $G\setminus G^o$ is the union of countably many singletons and conic arcs is the numerical range of some operator acting on a separable $\mathcal H$.

On the positive side, Anderson's theorem generalizes quite naturally to the infinite dimensional case under some restrictions on the operators involved. As was shown more recently in \cite{GauWu06},
the following result holds:
\begin{theorem} \label{th:GauWu}If $A$ is a compact operator on a Hilbert space with $W(A)$ contained in $\overline{\mathbb D}$ and $\overline{W(A)}$ intersecting $\mathbb T$ at infinitely many points, then
$W(A)=\overline{\mathbb D}$. \end{theorem}
In this paper, we single out a wider class of operators for which analogs of Anderson's theorem are valid in an infinite dimensional setting.
\section{Main results}\label{s:main}

We start with a lemma.
\begin{lemma} \label{l:gen}Let $A=N+K$, where $N$ is normal and $K$ is a compact operator on a Hilbert space $\mathcal H$. If $W(A)\subset\overline{\mathbb D}$ and $\gamma$ is a closed arc of $\mathbb T$ such that the intersection $\gamma\cap\overline{W(A)}$ is infinite while $\gamma\cap\sigma_{ess}(A)=\emptyset$, then
$\gamma\subset W(A)$. \end{lemma}
Recall that the essential spectrum $\sigma_{ess}(A)$ of an operator $A$ is the set of $\lambda\in\mathbb C$ such that the operator $A-\lambda I$ is not Fredholm. Equivalently, $\sigma_{ess}(A)$ is the spectrum of the
equivalence class of $A$ in the Calkin algebra of the algebra of bounded linear operators by the ideal of compact operators.

The proof of this lemma is delegated to the next section; we will discuss here some of its consequences.
\begin{theorem}\label{th:Gamma} Let $A=N+K$, where $N$ is normal and $K$ is a compact operator on a Hilbert space $\mathcal H$. Let also $W(A)\subset\overline{\mathbb D}$ and $\Gamma$ be a (relatively) open subset of $\mathbb T$ disjoint with $\sigma_{ess}(A)$. If every connected component of $\Gamma$ contains limit points of its intersection with $\overline{W(A)}$, then $\Gamma\subset W(A)$.  \end{theorem}
{\em Proof.} Connected components of $\Gamma$ are open arcs $\Gamma_j$. Writing $\Gamma_j$ as $\bigcup_{k=1}^\infty\gamma_{jk}$, where \[ \gamma_{j1}\subset\gamma_{j2}\subset\cdots\subset\gamma_{jk}\subset\cdots \] is an expanding family of closed arcs, we see that $\gamma=\gamma_{jk}$ satisfy the conditions of Lemma~\ref{l:gen} and thus $\gamma_{jk}\subset W(A)$, for $k$ large enough. Consequently, \[ \pushQED{\qed} \Gamma=\cup_{j,k=1}^\infty\gamma_{jk}\subset W(A).\qedhere \popQED \]
\begin{cor}\label{co:dense} Let $A$ and $\Gamma$ satisfy the conditions of Theorem~\ref{th:Gamma}, and in addition $\Gamma$ is dense in $\mathbb T$. Then \eq{incl}
{\mathbb D}\cup\Gamma\subset W(A)\subset \overline{W(A)}=\overline{\mathbb D}.\en \end{cor}
\begin{proof}By Theorem~\ref{th:Gamma} we have $\Gamma\subset W(A)$, and so $\conv\Gamma\subset W(A)$ due to the convexity of the numerical range. But $\Gamma$ being dense in $\mathbb T$ implies that
$\conv\Gamma\supset\mathbb D$. This proves the left inclusion in \eqref{incl}. The right equality then follows by combining ${\mathbb D}\subset W(A)$ with the given $W(A)\subset\overline{\mathbb D}$. \end{proof}
If the normal component $N$ of $A$ is in fact hermitian, then $\sigma_{ess}(A)\subset\mathbb R$. Choosing $\Gamma={\mathbb T}\setminus\{1,-1\}$ immediately yields
\begin{cor}\label{co:her}Let $A=H+K$, where $H$ is hermitian and $K$ is a compact operator on a Hilbert space $\mathcal H$. If $W(A)\subset\overline{\mathbb D}$ and the set $\overline{W(A)}\cap\mathbb T$ has limit points
both in the upper and lower open half plane, then $\overline{W(A)}=\overline{\mathbb D}$ and $\overline{W(A)}\setminus W(A)\subset\{1,-1\}$. \end{cor}
The next statement also is an immediate consequence of Corollary~\ref{co:dense}; we nevertheless state it as a theorem.
\begin{theorem}\label{th:cir}Let $A=N+K$, where $N$ is normal and $K$ is a compact operator on a Hilbert space $\mathcal H$. If $W(A)\subset\overline{\mathbb D}$ and the intersection ${\mathbb T}\cap\overline{W(A)}$ is infinite while $\sigma_{ess}(A)\subset\mathbb D$, then $W(A)=\overline{\mathbb D}$. \end{theorem}
\begin{proof}Indeed, $A$ satisfies the conditions of Corollary~\ref{co:dense} with $\Gamma=\mathbb T$, and so the inclusions in \eqref{incl} turn into the equalities. \end{proof}

\section{Proof of Lemma~\ref{l:gen}}

Note that the essential spectrum is invariant under addition of compact summands, and so $\sigma_{ess}(A)=\sigma_{ess}(N)$. The latter coincides with $\sigma(N)$ from which the isolated eigenvalues of finite multiplicity were removed. If $A$ is compact, that is, $N=0$, then of course $\sigma_{ess}(A)=\{0\}$, and condition $\sigma_{ess}(A)\subset\mathbb D$ holds. So, Theorem~\ref{th:GauWu} is a particular case of Theorem~\ref{th:cir} which was derived in the previous section from Lemma~\ref{l:gen}. On the other hand, our proof of Lemma~\ref{l:gen} below follows the lines of Gau-Wu’s proof of Theorem~\ref{th:GauWu}.

Let $d_A(\theta)=\sup W\left(\re(e^{-i\theta}A)\right)$, $\theta\in\mathbb R$, where as usual $\re X$ denotes the hermitian part $(X+X^*)/2$ of the operator $X$.
Since $d_A$ is the support function of the convex set $\overline{W(A)}$, condition $W(A)\subset\overline{\mathbb D}$ is equivalent to \eq{dA1} d_A(\theta)\leq 1, \quad \theta\in\mathbb R,\en
while the condition imposed on $\gamma\cap\overline{W(A)}$ means that the set
\eq{dA2} \alpha:=\{ e^{i\theta}\in\gamma \colon d_A(\theta)=1\} \en
is infinite. Consequently, $\alpha'\neq\emptyset$.

Observe now that for operators $A$ of the form $N+K$ the essential spectrum coincides with their Weyl spectrum $\omega(A)$, that is, the set of $\lambda$ for which $A-\lambda I$ is not a Fredholm operator with index zero.
By Berberian's spectral mapping theorem \cite[Theorem 3.1]{Ber70}, for any normal operator $T$ and a function $f$ continuous on $\sigma(T)$, $\omega(f(T))=f(\omega(T))$. Since $e^{-i\theta}A=e^{-i\theta}N+e^{-i\theta}K$
is the sum of a normal and compact operator along with $A$, we have
\begin{multline*} \sigma_{ess}\left(\re(e^{-i\theta}A)\right)= \omega\left(\re(e^{-i\theta}A)\right)=\re\left(\omega(e^{-i\theta}A)\right) \\
=\re\left(\sigma_{ess}(e^{-i\theta}A)\right)=\re\left(e^{-i\theta}\sigma_{ess}(A)\right). \end{multline*}
So, the condition $\gamma\cap\sigma_{ess}(A)=\emptyset$ implies that  \[ 1\in\sigma\left(\re(z^{-1}A)\right)\setminus\sigma_{ess}\left(\re(z^{-1}A)\right)
\text{ for all } z\in\alpha. \] In other words, $1$ is an isolated eigenvalue of $\re(z^{-1}A)$ of finite multiplicity whenever $z\in\alpha$.

As in \cite{GauWu06}, we now invoke \cite[Theorem 3.3]{Nar} according to which the points $z\in\alpha$ possess the following property: there exists a neighborhood $U_z$ of such $z$ and two (possibly coinciding) open analytic arcs $\gamma_j(z)\ni z$, $j=1,2$
satisfying \eq{inc}  \partial W(A)\cap U_z\subset \gamma_1(z)\cup\gamma_2(z)\subset W(A). \en For $z\in\alpha'$ we have in addition that at least one of the arcs $\gamma_j(z)$ contains infinitely many points of
the unit circle and thus lie in $\mathbb T$. Say for definiteness, $\gamma_1(z)\subset\mathbb T$. Since ${\overline{\mathbb D}\supset}W(A)\supset\gamma_1(z)$,
in fact the whole arc $\gamma_1(z)$ is a subset of $\alpha$, implying that $z$ is an interior point of $\alpha'$. So, $\alpha'$ is not only closed but also open in $\gamma$, and thus $\alpha'=\gamma$. So, $\alpha=\gamma$ as well.
Inclusions \eqref{inc} imply in particular that $\alpha\subset W(A)$, thus completing the proof. \qed

\section{Additional observations}

{\bf 1.} As in \cite{GauWu06}, the results of Section~\ref{s:main} remain valid with $\mathbb D$ and $\mathbb T$  replaced by an arbitrary elliptical disk and its boundary, respectively.
In order to see that, it suffices to consider a suitable affine transformation $\alpha A+\beta A^*+\gamma I$ of $A$ in place of $A$ itself.

{\bf 2.} Recall that Theorem~\ref{th:cir} is a generalization of Theorem~\ref{th:GauWu} from the case of compact $A$ to $A$ being the sum of a normal and compact summands under the additional condition
$\sigma_{ess}(A)\cap{\mathbb T}=\emptyset$. The following examples show that merely the condition on $\sigma_{ess}(A)$ would not suffice.

{\sl Example 1.}  Consider the $2$-by-$2$ matrix $C=\begin{bmatrix} 1/2 & 1 \\ 0 & 1/2\end{bmatrix}$ for which $\sigma(C)=\{1/2\}$ and $W(C)$ is the closed disk $E$ centered at $1/2$ with the radius also equal $1/2$.
In particular, $1\in E\subset\overline{\mathbb D}$.

Let now $Z$ be a countable subset of $\mathbb T$, and $A=\displaystyle\bigoplus_{z\in Z} zC$. For any $\lambda\notin\frac{1}{2}\overline{Z}$ we then have
\[ (A-\lambda I)^{-1}=\bigoplus_{z\in Z} \begin{bmatrix} \left(\frac{z}{2}-\lambda\right)^{-1} & -z\left(\frac{z}{2}-\lambda\right)^{-2} \\ 0 & \left(\frac{z}{2}-\lambda\right)^{-1}\end{bmatrix}, \]
and so \[ \sigma_p(A)=\frac{1}{2}Z\subset \sigma(A)=\frac{1}{2}\overline{Z}\subset\frac{1}{2}\mathbb T, \]
implying that $\sigma(A)$ is disjoint with $\mathbb T$. At the same time \[ Z\subset W(A)=\conv\{ zE\colon z\in Z\}\varsubsetneq\overline{\mathbb D}. \]
Moreover, by choosing $Z$ located on a sufficiently small arc it is possible to arrange for a sector in $\mathbb D$ disjoint with $W(A)$ and having an opening arbitrarily close to $\pi$.

{\sl Example 2.} Let now $S$ be a weighted shift, that is, $Se_j=s_je_{j+1}$, where $\{e_j\}_{j=1}^\infty$ is an orhtonormal basis of  $\mathcal H$, and $\{s_j\}$ is a bounded sequence.
It is well known (and easy to see) that both the numerical range $W(S)$ and the spectrum $\sigma(S)$ are invariant under rotation, and depend only on the absolute values of $s_j$ and not their arguments.
So, without loss of generality let us suppose that $s_j\geq 0$. Being convex, $W(S)$ is then either an open or a closed circular disk, while $\sigma(S)$ is a (naturally, closed) circular disk according to
e.g. \cite[Problem 93]{Hal82}.

Suppose in addition that the sequence  $\{s_j\}$ is periodic, say with the period $r$. Then $W(S)$ is open \cite[Proposition 6]{St83}, while its radius (coinciding in this case with the
numerical radius $w(S)$ of the operator $S$) is given by
\[ w(S)=\max\{\sum_{j=1}^r s_jx_jx_{j+1}\colon x_j\in{\mathbb R}, \sum_{j=1}^r x_j^2=1, x_{r+1}=x_1\} \]
\cite[Theorem 1]{Ridge76}. In particular, $w(S)\geq (s_1+\cdots s_r)/r$.
On the other hand, the spectral radius $r(S)$ of $S$ is the geometric mean $\sqrt[r]{s_1\cdots s_r}$ of the weights $s_1,\ldots,s_r$ \cite[Corollary 2]{Ridge70}. So, $r(S)<w(S)$, unless all the weights $s_j$ are the same.

By an appropriate scaling, we may arrange for $w(S)=1$ and thus $W(S)={\mathbb D}\neq\overline{\mathbb D}$, in spite of $\sigma(S)\subset\mathbb D$ being disjoint with $\mathbb T$.

\end{document}